\documentclass[letterpaper, 10pt, conference]{ieeeconf}

\IEEEoverridecommandlockouts                              	
\usepackage{graphics} 
\usepackage{epsfig} 
\usepackage{amsmath} 
\usepackage{amssymb}  
\usepackage{amsfonts}
\usepackage{url,bm,xspace,dsfont}
\usepackage{flushend}

\usepackage[usenames,dvipsnames]{color}
\usepackage{verbatim}

\newcommand{\TC}[1]{\textcolor{black}{{#1}}}

\def\proj{\mathop{ {\bf \Pi}_{{\cal H}_{0,t}^{y^i}}}} 

\newtheorem{theorem}{Theorem}
\newtheorem{lemma}{Lemma}
\newtheorem{definition}{Definition}
\newtheorem{corollary}{Corollary}

\newtheorem{assumptions}{Assumptions}
\newtheorem{problem}{Problem}

\newcommand{\veps}{\varepsilon}
\newcommand{\la}{\langle}
\newcommand{\ra}{\rangle}

\newcommand{\sr}{\stackrel}

\newcommand{\rar}{\rightarrow}

\newcommand{\tri}{\sr{\triangle}{=}}

\newcommand{\be}{\begin{equation}}
\newcommand{\ee}{\end{equation}}
\newcommand{\bea}{\begin{eqnarray}}
\newcommand{\eea}{\end{eqnarray}}
\newcommand{\bes}{\begin{eqnarray*}}
\newcommand{\ees}{\end{eqnarray*}}

\newcommand{\bi}{\begin{itemize}}
\newcommand{\ei}{\end{itemize}}
\newcommand{\ben}{\begin{enumerate}}
\newcommand{\een}{\end{enumerate}}

\newcommand{\bp}{\begin{problem}}
\newcommand{\ep}{\end{problem}}
\newcommand{\hso}{\hspace{.1in}}
\newcommand{\hst}{\hspace{.2in}}

\newcommand{\noi}{\noindent}

\newcommand{\bc}{\begin{center}}
\newcommand{\ec}{\end{center}}

\begin{document}
%
%
%
%
\title{\bf Team and Person-by-Person Optimality Conditions of Differential Decision Systems}


\author{ Charalambos D. Charalambous\thanks{C.D. Charalambous and C.N. Hadjicostis are with the Department of Electrical and Computer Engineering, University of Cyprus, Nicosia 1678 (E-mails:{\tt \{chadcha,chadjic\}@ucy.ac.cy}).}, Themistoklis Charalambous\thanks{T. Charalambous is with  with the Automatic Control Lab, Electrical Engineering Department and ACCESS Linnaeus Center, Royal Institute of Technology (KTH), Stockholm, Sweden.  Corresponding author's address: Osquldas v\"{a}g 10, 100-44 Stockholm, Sweden (E-mail: {\tt themisc@kth.se}).} and Christoforos N. Hadjicostis
}

\maketitle
%
%
%
%
\begin{abstract}
In this paper,  we derive team and person-by-person optimality conditions for distributed differential decision  systems with different or decentralized information structures. The necessary conditions of optimality are given in terms of Hamiltonian system of equations  consisting of a coupled backward and forward differential equations and a Hamiltonian projected onto the subspace generated by the decentralized information structures. Under certain global convexity conditions it is shown that the optimality conitions are also sufficient.

\end{abstract}

%
%
%
%
\section{Introduction}\label{introduction}


When the system model consists of multiple decision makers,  and  the acquisition of information and its processing is decentralized or shared among several locations,  the decision makers  actions are based on different information.  We call  the information available for such  decisions,  \emph{``decentralized information structures or patterns''}.  When the system model is dynamic, consisting of an interconnection of at least two subsystems, and  the decisions are based on decentralized information structures, we call the overall system  a \emph{``distributed system with decentralized information structures''}.

 Over the years several specific forms of decentralized information structures are analyzed mostly in discrete-time (see, for example \cite{bamieh-voulgaris2005,nayyar-mahajan-teneketzis2011,vanschuppen2011,lessard-lall2011,mahajan-martins-rotkowitz-yuksel2012} for the most recent approaches). However, at this stage the systematic framework  addressing  optimality conditions for distributed systems with decentralized information structures is \cite{charalambous-ahmedFIS_Parti2012,charalambous-ahmedFIS_Partii2012}, where necessary and sufficient team game optimality conditions are given for distributed stochastic differential systems with decentralized information structures. 

In this paper, we draw the corresponding results for deterministic continuous- and discrete-time systems with decentralized information structures. More specifically, we consider a team game reward (e.g., \cite{marschak1955,radner1962,waal-vanschuppen2000})  and we apply concepts from  the classical theory of optimization to derive necessary and sufficient optimality conditions for  nonlinear distributed systems with decentralized information structures. The optimality conditions developed in this paper can be applied to  many  architectures of distributed systems (see, for example, Fig.~\ref{ADC}). 
\noi The specific contributions of this paper  are the following. \\
\noi \textbf{(a)} Derive team games necessary and sufficient conditions of optimality for distributed deterministic differential decision systems with decentralized information structures. \\
\noi \textbf{(b)} Derive person-by-person optimality conditions and discuss their relation with team optimality conditions;\\
\noi \textbf{(c)} Apply the optimality conditions to cetrain types of differential team games.

\begin{figure}[t]
\begin{center}
\includegraphics[width=\columnwidth]{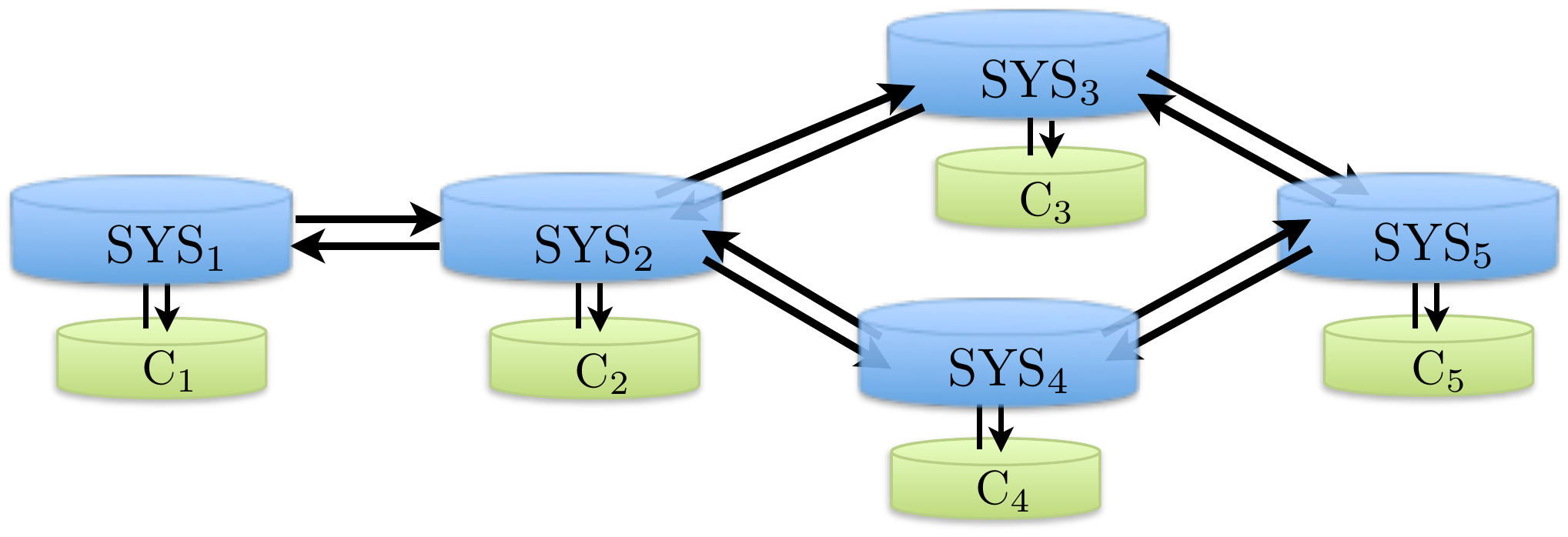}
\caption{Diagram of an example of the architecture for distributed decision systems.}
\label{ADC}
\end{center}
\end{figure}

In Section~\ref{sec:notation} the notation used throughout the paper is provided, along with some background on team games and information structures that is needed for our subsequent development. In Section~\ref{sec:deterministic}, we first introduce  the  formulation of the team and person-by-person decision problems of differential systems, and then we derive the optimality conditions. In Section~\ref{sec:examples}, we compute the optimal strategies for specific pay-off and differential structures and in Section~\ref{sec:discrete}, we provide the equivalent formulation for discrete-time dynamical systems.

%
%
%
%
\section{Notation and Preliminaries} \label{sec:notation}


The sets of real, integer and natural numbers are denoted by $\mathds{R}$, $\mathds{Z}$ and $\mathds{N}$, respectively; ${\mathbb Z}_N  \tri \{1,2,\ldots, N\} $ and ${\mathbb Z}^0_N  \tri \{0,1,2,\ldots, N\} $.  
The Borel algebra on $[0,T]$ is denoted by ${\cal B}([0,T])$ and the linear transformation mapping of a vector space ${\cal X}$ into a vector space ${\cal Y}$ is denoted by ${\cal L}({\cal X},{\cal Y})$.  $\la a, b\ra $ represents the inner product in ${\mathbb R}^n, \forall a, b \in {\mathbb R}^n$ for some positive integer $n$, whereas  $| a|_{{\mathbb R}^n} \tri \sqrt{\la a, b\ra} $ is the norm on ${\mathbb R}^n, \forall a\in {\mathbb R}^n$ for some positive integer $n$. ${\cal H} = M \bigoplus M^\perp$ is a direct sum representation of a Hilbert space ${\cal H}$, where $M$ is a closed subspace of ${\cal H}$ and $M^\perp$ its orthogonal complement.
 ${\bf \Pi}_M(x)$ is the orthogonal projection of a Hilbert space element $x\in {\cal H}$ onto the subspace $M \subset {\cal H}$. 

Our derivations will make use of the following spaces. 
$C([0,T], {\mathbb R}^n) \tri \Big\{\mbox{continuous functions} \: \phi: [0,T] \longrightarrow {\mathbb R}^n :   \sup_{t \in [0,T]} |\phi(t)|_{{\mathbb R}^n} < \infty \Big\}$;
$B^\infty([0,T], {\mathbb R}^n) \tri \Big\{ \mbox{measurbale functions} \: \phi: [0,T] \longrightarrow {\mathbb R}^n :  ||\phi||^2 \tri  \sup_{t \in [0,T]}  |\phi(t)|^2_{{\mathbb R}^n} < \infty \Big\}$.
\noi For Lebesgue measurable functions $ \phi$ we have the following spaces:
$ L^2([0,T], {\mathbb R}^n) \tri \Big\{\phi: [0,T] \longrightarrow {\mathbb R}^n: 
\int_{[0,T]} |z(t)|_{{\mathbb R}^n}^2 dt < \infty \Big\}$ ,
$ L^2([0,T],   {\cal L}({\mathbb R}^m,{\mathbb R}^n)) \tri \Big\{\phi: [0,T] \longrightarrow {\mathbb R}^{n \times m} : 
\int_{[0,T]} |\Sigma(t)|_{{\cal L}({\mathbb R}^m,{\mathbb R}^n)}^2 dt  \tri  \int_{[0,T]} tr(\Sigma^*(t)\Sigma(t)) dt < \infty  \Big\}$.

%
%
%
%
\section{Team Games of Differential Systems}\label{sec:deterministic}

We first introduce  the mathematical formulation of the team and person-by-person (PbP) decision problems of differential systems, and then we derive the optimality conditions. We invoke decision maker (DM) strategies which are  deterministic measurable  functions, also known as regular  strategies.

\subsection{Elements of Team Games}\label{elements}

The basic elements of a team game are the state space, the observation space,  the DMs action spaces, and the pay-off. These are described  below.\\

\noi{\bf Unobserved State Space}\\
The unobserved state space is assumed to be a  linear complete separable metric space $({\mathbb X}_{[0,T]},d)$, where  ${\cal B}({\cal X}_{[0,T]})$ are the  measurable subsets of the unobserved state space ${\mathbb X}_{[0,T]}$ generated by open sets (with respect to metric $d$). Elements $x \in {\mathbb X}_{[0,T]}$ are the unobserved state trajectories.  Since state trajectories are solutions of differential equations, an envisioned scenario is ${\mathbb X}_{[0,T]} =C(0,T], {\mathbb R}^n)$, where ${\cal B}({\cal X}_{[0,T}])={\cal B}(C([0,T], {\mathbb R}^n))$ is the $\sigma-$field generated by cylinder sets in $C([0,T], {\mathbb R}^n)$,   and a state trajectory is  $x \tri \{x(t): t \in [0,T]\} \in C([0,T], {\mathbb R}^n)$.   We also introduce $\sigma-$field generated by truncations of $x \in C([0,T], {\mathbb R}^n)$ defined by  
\begin{align}
{\cal B}_t(&C[0,T], {\mathbb R}^n))  \tri  \sigma \Big\{ \{ x \in C([0,T], {\mathbb R}^m): x(s) \in A \}: \nonumber \\
& 0\leq s \leq t, \hso A \in {\cal B}({\mathbb R}^n) \Big\}, \  t \in [0,T].
\end{align}
Thus, $\{ {\cal B}_t(C[0,T], {\mathbb R}^n)) : t \in [0,T]\}$ is a family of $\sigma-$fields which is nondecreasing, \\${\cal B}_s(C[0,T], {\mathbb R}^n)) \subseteq 
{\cal B}_t(C[0,T], {\mathbb R}^n)), 0 \leq s \leq t \leq T$. 
Thus, for continuous trajectories the space $C([0,T], {\mathbb R}^n)$ represent the unobserved state space, and its elements the unobserved state trajectories. \\

\noi{\bf Observation Space}\\
The observation space is assume to be a linear complete separable metric space $({\mathbb Y}_{[0,T]}^i, d^i)$, where ${\cal B}({\cal Y}_{[0,T]}^i)$ are its   measurable subsets generated by open sets, for $i=1, \ldots, N$. Thus, elements $y \in {\mathbb Y}_{[0,T]}^i$ represent the observable trajectories. For unobserved state space $C([0,T], {\mathbb R}^n)$, the observable trajectories are generated by the maps 
\begin{align}
h^i: [0,T] \times C([0,T], {\mathbb R}^n)  \longrightarrow& {\mathbb R^{k_i}}, \ y^i(t)\tri h^i(t,x), \nonumber \\ 
&\hst i=1, \ldots N, \label{obs1}
\end{align}
such that $\{h^i(t.x): (t,x) \in [0,T] \times C([0,T], {\mathbb R}^{k_i})  \}$ have the following property: for all $t \in [0,T]$, the map $(s,x) \rar h^i(s,x)$ is  $ {\cal B}([0,t]) \otimes   {\cal B}_t(C[0,T], {\mathbb R}^n))/{\cal B}({\mathbb R}^{k_i})-$measurable for $i=1, \ldots, N$. When this propery holds we say, $\{y^i(t): t \in [0,T]\}$ is progressively measurable with respect to the family $\{ {\cal B}_t(C[0,T], {\mathbb R}^{k_i})) : t \in [0,T]\}$. Often we shall assume observation trajectories which are square integrable $y^i \in L^2([0,T], {\mathbb R}^{k_i})$, and progressively measurable with respect to $\{ {\cal B}_t(C[0,T], {\mathbb R}^{k_i})) : t \in [0,T]\}$. Given an underlying Hilbert space ${\cal H}^i$  we denote by ${\cal H}_{0,t}^{y^i} \tri \overline{Span}\Big\{y^i(s): 0\leq s \leq t\Big\}$ the closed subspace generated by $\{y^i(s): 0 \leq s \leq t\}$ which is an element of the  Hilbert space (i.e.,  ${\cal H}_{0,t}^{y^i}  \subset {\cal H}^i$), $t \in [0,T],$ for $i=1, \ldots, N$. Note that the above constructions also embeds as a special case observation trajectories which are independent of $x$, by setting $y^i(t)=h^i(t), h^i: [0,T] \longrightarrow {\mathbb R}^{k_i}$, for $i=1, \ldots, N$. \\

\noi{\bf Team Members}\\
The team is assumed to consist of  $N$ Decision Makers (DM) or players whose actions $\{u_t^i: t \in [0,T]\}$, take values in a closed convex set ${\mathbb A}^i$ of  linear separable  metric space $({\mathbb M}^i,d^i), i=1, \ldots, N$. Unlike the centralized decision making, each DMs $i$ actions depends only  on his own observation space $({\mathbb Y}_{[0,T]}^i, d)$. Let  $\{{\cal B}_t({\cal Y}_{[0,T]}^i): t \in [0,T]\}$ denote the family of $\sigma-$fields generated by truncations of $y^i \in {\mathbb Y}_{[0,T]}^i,  i=1, \ldots N$.   The set of admissible laws or strategies of DM $i$, denoted by  $ {\mathbb U}_{reg}^i[0, T]$, is  defined by 
\footnote{We often write $ L_{{\cal Y}_T^i}^2([0,T],{\mathbb R}^{d_i}) \equiv  L^2([0,T],{\mathbb R}^{d_i})$ to indicate that its elements are   $\{{\cal B}_t({\cal Y}_{[0,T]}^i): t \in [0,T]\}-$progressively measurable.}  
 \begin{align}
 {\mathbb U}&_{reg}^i[0, T] \tri \Big\{  u^i   \in  L^2([0,T],{\mathbb R}^{d_i})  : \:   u_t^i \in {\mathbb A}^i \subset {\mathbb R}^{d_i}, \: \nonumber \\ 
 & t \in [0,T],  \: u^i \: {is} \:\{{\cal B}_t({\cal Y}_{[0,T]}^i): t \in [0,T]\} \nonumber \\
 &  \hst-\mbox{progressively measurable} \Big\}, \hso  \forall  i \in {\mathbb Z}_N. \label{cs1a}
 \end{align}
Clearly,   ${\mathbb U}_{reg}^i[0, T]$ is a  closed convex subset of  $L_{{\cal Y}_T^i}^2([0,T],{\mathbb R}^{d_i})$, for $i=1,2, \ldots, N$.  
 The set of admissible $N$ team or person-by-person strategies    is denoted by ${\mathbb U}_{reg}^{(N)}[0,T] \tri \times_{i=1}^N {\mathbb U}_{reg}^i[0,T]$. 
  
\noi  The DM actions $\{u_t^i: t \in [0,T]\}$ are called: \\
\noi {\bf Open Loop (OL),} if  $u_t^i=\mu^i(t)$, for $t \in [0,T]$, where $\mu^i: [0,T] \longrightarrow {\mathbb A}^i$ are deterministic measurable functions, $i=1, \ldots, N$; \\
\noi {\bf Closed Loop Feedback (CLF),}  if  $u_t^i=\mu^i(t,y^i)$ are nonanticipative functionals of the observation trajectory $y^i(\cdot)$, for $t \in [0,T]$, where $\mu^i: [0,T] \times {\mathbb Y}_{i[0,T]}^i \longrightarrow {\mathbb A}^i$, are deterministic measurable mappings, $i=1, \ldots, N$; \\
\noi {\bf Closed Loop Markov (CLM),}  if  $u_t^i=\mu^i(t,y^i(t))$, for $t \in [0,T]$, where $\mu^i: [0,T] \times {\mathbb R}^{k_i}\longrightarrow {\mathbb A}^i$, are deterministic measurable functions, $i=1, \ldots, N$.

Clearly, open loop strategies can be described via observations $\{y^i(t): t \in [0,T]\}$ which belong to closed subspaces generated by finite number of basis, ${\cal H}_{0,t}^{y^i} \tri Span\Big\{e_1^1, e_2^i, \ldots, e_{j_i}^i\Big\}$ of a Hilbert space ${\cal H}_{0,t}^i, t\in [0,T]$, for $i=1, \ldots, N$. \\

 \noi {\bf Distributed Differential  System}\\
 A distributed differential system consists of an interconnection of $N$ subsystems.  Each subsystem $i$ has its own   state vector ${\mathbb R}^{n_i}$, action space ${\mathbb A}^i \subset {\mathbb R}^{d_i}$, and an  initial state vector $x^i(0)=x_0^i$, described by a  system of coupled  differential equations  as follows.
 \begin{align}
 \dot{x}^i(t) &=f^i(t,x^i(t),u_t^i)  + \sum_{j=1, j \neq i}^N f^{ij}(t,x^j(t),u_t^j) \;, \nonumber \\ 
& \hst x^i(0) = x_0^i, \hso  t \in (0,T], \hso i \in {\mathbb Z}_N. \label{eq1ds}
 \end{align}
Define the augmented vectors by
\bes
  u \tri (u^1, u^2, \ldots, u^{N}) \in {\mathbb R}^d, \hso x \tri (x^1, x^2, \ldots, x^{N}) \in {\mathbb R}^n.
 \ees
In compact form the distributed  differential  system  is described  by
 \begin{eqnarray}
 \dot{x}(t) =  f(t,x(t),u_t), \hst x(0) = x_0, \hst  t \in (0,T], \label{eq1}
 \end{eqnarray}
 where $f: [0,T] \times {\mathbb R}^n\times {\mathbb A}^{(N)} \longrightarrow {\mathbb R}^n$.  Note that \eqref{eq1} is very general since no specific interconnection structure is assumed among the different subsystems. \\

   \noi {\bf Pay-off Functional}\\
Consider the distributed system (\ref{eq1}) with a given admissible set of DMs strategies.  
  \noi Given a $u \in {\mathbb U}_{reg}^{(N)}[0,T],$ we  define the reward or performance criterion by
\begin{align}
J(u^1,\ldots,u^N) \tri \int_{0}^{T}  \ell(t,x(t),u_t) dt + \varphi(x(T),            \label{cfd}
 \end{align}
  where $\ell: [0,T] \times {\mathbb R}^n\times {\mathbb U}^{(N)} \longrightarrow (-\infty, \infty]$   and $\varphi : {\mathbb R}^n \longrightarrow (-\infty, \infty]$.
 Notice that the performance of the decentralized system is measured  by  a single pay-off functional. The underlying assumption concerning the single pay-off instead of  multiple pay-offs  (one for each decision maker) is that the team objective   can be met.\\

\noi {\bf Team and Person-by-Person Optimality}\\
Given the basic elements of the team game introduce above, we now introduce the   definitions of team and Person-by-Person (PbP) or (player-by-player)  optimality.  \\

 \begin{problem}(Team and Person-by-Person Optimality)
 \label{problem1}
\noi {\bf (T): Team Optimality.} Given the  pay-off functional (\ref{cfd}),   constraint (\ref{eq1})   the  $N$ tuple of  strategies   $u^o \tri (u^{1,o}, u^{2,o}, \ldots, u^{N,o}) \in {\mathbb U}_{reg}^{(N)}[0,T]$  is called  team optimal if it satisfies
 \bea
 J(u^{1,o}, u^{2,o}, \ldots, u^{N,o}) \leq J(u^1, u^2, \ldots, u^N),  \label{cfd1ad}
 \eea
 for all $u\tri (u^1, u^2, \ldots, u^N) \in {\mathbb U}_{reg}^{(N)}[0,T]$. Any $u^o   \in {\mathbb U}_{rel}^{(N)}[0,T]$ satisfying (\ref{cfd1ad}) is called an optimal  regular decision strategy (or control) and the corresponding $x^o(\cdot)\equiv x(\cdot; u^o(\cdot))$ (satisfying (\ref{eq1})) the  optimal state process. \\
\noi {\bf (PbP): Person-by-Person Optimality.} Given the   pay-off functional (\ref{cfd}),   constraint (\ref{eq1})   the  $N$ tuple of  strategies   $u^o \tri (u^{1,o}, u^{2,o}, \ldots, u^{N,o}) \in {\mathbb U}_{reg}^{(N)}[0,T]$  is called  person-by-person optimal  if it satisfies
\begin{align}
 \tilde{J}(u^{i,o}, u^{-i,o}) = J(u^o) \leq \tilde{J}(u^{i}, u^{-i,o}),  \label{cfd2d}
 \end{align}
for all $u^i \in {\mathbb  U}_{reg}^i[0,T], \hso \forall i \in {\mathbb Z}_N$, where 
 \bes
 \tilde{J}(v,u^{-i}) \tri J(u^1,u^2,\ldots, u^{i-1},v,u^{i+1},\ldots,u^N).
 \ees
\end{problem}

 Conditions (\ref{cfd2d}) are analogous to the Nash equilibrium strategies of team games consisting of a single pay-off and $N$ DM.   The rationale for the restriction to PbP optimal strategy is based on the fact that the actions of the $N$ DM are not communicated to each other, and hence they cannot do better than restricting attention to this optimal strategy.

\subsection{Existence of Solutions and Continuous Dependence}
\label{ecd}
Herein,  we study the question of existence of solutions to \eqref{eq1}  and its continuous dependence on the DM strategies based on the following assumptions.

\begin{assumptions}
The drift $f$ associated with \eqref{eq1}  is a   Borel measurable  map defined by
\begin{eqnarray*}
 f: [0,T] \times {\mathbb R}^n \times {\mathbb A}^{(N)} \longrightarrow {\mathbb R}^n,
 \end{eqnarray*} 
and there exists a $K \in L^{2,+}([0,T], {\mathbb R})$ such that \\
\textbf{\emph{(A1)}} $|f(t,x,u)-f(t,y,u)|_{{\mathbb R}^n} \leq K(t) |x-y|_{{\mathbb R}^n}$ uniformly in $u \in {\mathbb A}^{(N)}$; \\
\textbf{\emph{(A2)}} $|f(t,x,u)-f(t,x,v)|_{{\mathbb R}^n} \leq K(t) |u-v|_{{\mathbb R}^d}$ uniformly in $x \in {\mathbb R}^n$; \\
\textbf{\emph{(A3)}} $|f(t,x,u)|_{{\mathbb R}^n} \leq K(t) (1 + |x|_{{\mathbb R}^n}+ |u|_{{\mathbb R}^d})$.
\label{A1-A4}
\end{assumptions}

 Assumptions~\ref{A1-A4}  are sufficient conditions for the existence of a unique $C([0,T], {\mathbb R}^n)$  solution which is also an element  of the space $B^{\infty}([0,T],{\mathbb R}^n)$.

\noi The following lemma establishes such results and  continuous dependence of solutions on the DM strategies.
\begin{lemma}
\label{lemma3.1}
Suppose Assumptions~\ref{A1-A4} hold. Then for any $u \in {\mathbb U}_{reg}^{(N)}[0,T]$,  the following hold. \\
 \noi {\bf 1)} System (\ref{eq1}) has a unique solution   $x \in B^{\infty}([0,T],{\mathbb R}^n)$  which is  continuous $x \in C([0,T],{\mathbb R}^n)$. \\
\noi {\bf 2)} The solution of  system  (\ref{eq1}) is continuously dependent on the DM strategies, in the sense that, as $u^{i, \alpha} \longrightarrow u^{i,o}$  in ${\mathbb U}_{reg}^i[0,T]$, $\forall i \in {\mathbb Z}_N$,  $x^\alpha \longrightarrow x^o $ in $B^{\infty}([0,T],{\mathbb R}^n)$.
\end{lemma}

\begin{proof}
\TC{Similar to \cite[Lemma 1]{charalambous-ahmedFIS_Parti2012}.}
\end{proof}

\subsection{Team and PbP Optimality Conditions}\label{toc}

For the  derivation of optimality conditions we shall require stronger regularity conditions for  $f$, as well as, for the running and terminal pay-offs functions $\{\ell,\varphi\}.$  These are given below.

\begin{assumptions}
\label{NCD1}
The maps of $\{f,\ell, \varphi\} $ satisfy the following conditions.\\
\noi \textbf{\emph{(B1)}} The map $f: [0,T] \times {\mathbb R}^n \times {\mathbb A}^{(N)} \longrightarrow {\mathbb R}^n$ is continuous in $(t,x,u)$ and continously differentiable with respect to $x,u$; \\
\noi \textbf{\emph{(B2)}} The first derivatives of $\{f_x, f_u \}$ are bounded  uniformly on $[0,T] \times {\mathbb R}^n \times {\mathbb A}^{(N)}$. \\
\noi \textbf{\emph{(B3)}} The maps $\ell: [0,T] \times {\mathbb R}^n \times {\mathbb A}^{(N)} \longrightarrow (-\infty, \infty]$ is Borel measurable, continuously differentiable with respect to  $(x,u)$, the map $\varphi: [0,T] \times {\mathbb R}^n \longrightarrow (-\infty, \infty]$ is continously differentiable with respect to $x$, $\ell(0,0,t)$ is bounded, and there exist $K_1, K_2 >0$ such that
\begin{align*}
|\ell_x(t,x,u)|_{{\mathbb R}^n}+|\ell_u(t,x,u) |_{{\mathbb R}^d} &\leq K_1 \big(1+|x|_{{\mathbb R}^n} + |u|_{{\mathbb R}^d} \big), \\ \hso |\varphi_x(x)|_{{\mathbb R}^n} &\leq K_2 \big(1+ |x|_{{\mathbb R}^n}\big).
\end{align*}
\noi \textbf{\emph{(B4)}} $|h^i(t,x)|_{{\mathbb R}^{k_i}} \leq K  \sup_{ 0\leq s \leq t}    \Big(1 + |x(s)|_{{\mathbb R}^n}^2\Big), \forall t \in [0,T], x \in C([0,T], {\mathbb R}^n),  i=1, \ldots, N$.
\end{assumptions}

Note that {\bf (B1), (B2)} imply that $|f(t,x,u)|_{{\mathbb R}^n}     \leq K \big(1+|x|_{{\mathbb R}^n} + |u|_{{\mathbb R}^d} \big), K>0$, and {\bf (B4)} implies $h^i \in B^\infty([0,T], {\mathbb R}^n), i=1, \ldots, N$.\\

First, we derive necessary conditions for team and PbP optimality. For this derivation, we need the so-called \emph{variational equation}.  We note that for differential systems, the strategies can be either open-loop or feedback, and feedback strategies do not give smaller pay-off. 
Thus, the minimum pay-off attainable under open loop strategies is equal to the minimum pay-off attainable under feedback strategies. This is well known in deterministic optimal control theory. The point to be made is that when considering variations in the state trajectory the DM strategies do not react so we do not need to introduce derivatives of the $u$ variable  with respect to the state. 

Suppose $u^o \tri (u^{1,o}, u^{2,o}, \ldots, u^{N,o}) \in {\mathbb U}_{rel}^{(N)}[0,T]$ denotes the optimal decision and $u \tri (u^1, u^2, \ldots, u^N) \in {\mathbb U}_{rel}^{(N)}[0,T]$ any other decision.  Since ${\mathbb A}^i$ is convex then ${\mathbb U}_{reg}^i[0,T]$ is convex $\forall i \in {\mathbb Z}_N$, it is clear that  for any $\varepsilon \in [0,1]$,
\bes
 u_t^{i,\varepsilon} \tri u_t^{i,o} + \varepsilon (u_t^i-u_t^{i,o}) \in {\mathbb U}_{reg}^i[0,T], \hst \forall i \in {\mathbb Z}_N.
 \ees
 Let $x^{\varepsilon}(\cdot)\equiv x^\veps(\cdot; u^\veps(\cdot))$ and  $x^{o}(\cdot) \equiv x^o(\cdot;u^o(\cdot))  \in B^{\infty}([0,T],{\mathbb R}^n)$ denote the solutions  of the system equation (\ref{eq1})  corresponding to  $u^{\varepsilon}(\cdot)$ and $u^o(\cdot)$, respectively.  Consider the limit
 \bea
  Z(t) \tri \lim_{\varepsilon\downarrow 0}  \frac{1}{\veps} \Big\{x^{\varepsilon}(t)-x^o(t)\Big\} , \hst t \in [0,T]. \label{veq}
  \eea
  
We have the following result characterizing the variational equation.
\begin{lemma}\label{lemma4.1}
Suppose Assumptions~\ref{NCD1} hold and consider strategies ${\mathbb U}_{reg}^{(N)}[0,T]$.  The process $\{Z(t): t \in [0,T]\}$ defined by (\ref{veq})  is an element of the Banach space $B^{\infty}([0,T],{\mathbb R}^n)$ and it  is the unique solution of the variational differential equation
 \begin{align}
 \dot{Z}(t) =& f_x(t,x^o(t),u_t^o)Z(t)   \label{eq9} \\
 &+ \sum_{i=1}^N f_{u^i}(t,x^o(t),u_t^{,o})(u_t^i-u_t^{i,o}),  \: Z(0)=0. \nonumber 
 \end{align}
 having trajectories $Z \in C([0,T], {\mathbb R}^n)$. 
  \end{lemma}

\begin{proof}
\TC{Similar to \cite[Lemma 2]{charalambous-ahmedFIS_Parti2012}.}
\end{proof}

Before we show the optimality conditions we define the Hamiltonian system of equations, i.e., 
\bes
 {\cal H}: [0, T] \times {\mathbb R}^n\times {\mathbb R}^n\times   {\mathbb A}^{(N)} \longrightarrow {\mathbb R}
\ees
  given  by
   \begin{align}
    { H} (t,x,\psi,u) \tri    \langle f(t,x,u),\psi \rangle  + \ell(t,x,u),  \ t \in  [0, T]. \label{h1}
    \end{align}
    \noi For any $u \in {\mathbb U}_{reg}^{(N)}[0,T]$, the adjoint process  $\psi \in  L^2([0,T], {\mathbb R}^n)$ satisfies the following backward  differential equation
\begin{subequations}
\begin{align}
\dot{\psi} (t) =& -f_x^{*}(t,x(t),u_t)\psi (t)  -\ell_x(t,x(t),u_t) \nonumber \\
=& - { H}_x (t,x(t),\psi(t),u_t),  \hso t \in [0,T). \label{adj1a} \\
\hst \psi(T) =&   \varphi_x(x(T)). \label{eq18}
 \end{align}
 \end{subequations}
In terms of the Hamiltonian, the state process satisfies the differential equation
\begin{align*}
\dot{x}(t) =&f(t,x(t),u_t)= { H}_\psi (t,x(t),\psi(t),u_t),  \  t \in (0,T] \\
        x(0) =&  x_0. 
 \end{align*}

Next, we state and prove the necessary conditions for team optimality. Specifically, given that $u^o \in {\mathbb U}_{reg}^{(N)}[0,T]$  is team optimal, we show that it leads naturally  to   the Hamiltonian system of equations (called necessary conditions).

\begin{theorem} [\textbf{Necessary conditions for team optimality}]\label{theorem5.1}
Consider Problem~\ref{problem1} under Assumptions~\ref{NCD1}, and assume ${\mathbb A}^i$ are closed, bounded and convex subsets of ${\mathbb R}^{d_i}$, and $\{y^i(s): 0\leq s \leq t\}$ generates  ${\cal H}_{0,t}^{y^i}$-a closed subspace of a Hilbert space for $ i=1, \ldots, N$.\\
For  an element $ u^o \in {\mathbb U}_{reg}^{(N)}[0,T]$ with the corresponding solution $x^o \in B^{\infty}([0,T],{\mathbb R}^n)$ to be team optimal, it is necessary  that
the following conditions  hold. \\
\noi \textbf{\emph{1)}} There exists  a process $\psi^o \in  L^2([0,T],{\mathbb R}^n)$. \\
\noi \textbf{\emph{2)}} The triple $\{u^o,x^o,\psi^o\}$ satisfy the  inequality:
  \begin{align}
  \sum_{i=1}^N \int_{0}^{T}  \la { H}_{u^i}(t,x^o(t)\psi^o(t)&, u_t^{o}),u_t^i-u_t^{i,o}\ra dt
  \geq 0,\nonumber \\
 & \hst \forall u \in {\mathbb U}_{reg}^{(N)}[0,T]. \label{eq16} 
\end{align}
\noi \textbf{\emph{3)}} The process ${\psi}^o$  is the  unique $C([0,T], {\mathbb R}^n)$ solution of the backward differential equation (\ref{adj1a}), (\ref{eq18}) and   $u^o \in {\mathbb U}_{reg}^{(N)}[0,T]$ satisfies  the  inequalities:
\begin{align}
   \la {\bf \Pi}_{{\cal H}_{0,t}^{y^i}}& \Big( H_{u^i}(t,x^o(t),\psi^o(t),u_t^{o})\Big) , v^i-u_t^{i,o}\ra   \geq 0, \nonumber \\
   &\hst \forall v^i \in {\mathbb A}^i,   t \in [0,T], i=1,2,\ldots, N.   \label{eqh35}
\end{align}
\end{theorem}

\begin{proof}
\TC{For \textbf{1)} and  \textbf{2)},  this is similar to \cite[Theorem 6]{charalambous-ahmedFIS_Parti2012}. For \textbf{3)}, consider  $\Big\{H_{u^i}(t,x^o, \psi^o(t), u_t^o): t \in [0,T]\Big\}$ lying in the Hilbert space of square integrable functions ${\cal H}^i, i=1, \ldots, N$, and the set of observables $\{y^i(t): t \in [0,T]\}$ generating a closed subspace ${\cal H}_{0,t}^{y^i} \tri \overline{\mbox{Span}}\Big\{y^i(s): 0 \leq s \leq t\Big\} \subset   {\cal H}^i, i\in {\mathbb Z}_N$. Then for any $H_{u^i} \in {\cal H}^i$ we have the decomposition
 \begin{align*}
 H_{u^i} &(t,x^o(t),\psi^o(t),u_t^{o}) \\
& = {\bf \Pi}_{{\cal H}_{0,t}^{y^i}}\Big(H_{u^i} (t,x^o(t),\psi^o(t),u_t^{o})\Big) + E(t), \\
& \: E(t) \perp {\cal H}_{0,t}^{y^i}, \: t \in [0,T],  i\in {\mathbb Z}_N.
\end{align*}
 Since $u_t-u_t^i \in  {\cal H}_{0,t}^{y^i}$, by substituting the above decomposition in (\ref{eq16}) we obtain
 \begin{align}
  \sum_{i=1}^N    \int_{0}^{T}  \la &{\bf  \Pi}_{ {\cal H}_{0,t}^{y^i}}\Big( { H}(t,x^o(t),\psi^o(t),u_t^{o})\Big), \nonumber \\
  &u_t^i-u_t^{i,o} \ra   dt   \geq 0, \hso \forall u \in {\mathbb U}_{reg}^{(N)}[0,T].  \label{eq36a}
\end{align}
\noi Let  $t \in (0,T),$ and $\varepsilon >0$, and  consider the set $I_{\varepsilon} \equiv [t,t+\varepsilon] \subset [0,T]$    such that $|I_{\varepsilon}| \rightarrow 0$   as $\varepsilon \rightarrow 0,$ for $i=1,2, \ldots, N$.  For any ${\cal H}_{0,t}^{y^i}-$progressively measurable   $v_t^i \in {\mathbb A}^i,$ construct
\bea
 u_t^i = \begin{cases}  v_t^i & ~ \mbox{for}~~ t \in I_{\varepsilon}  \\   u_t^{i,o} & \mbox{ otherwise}           \end{cases} \hst i=1,2, \ldots, N. \label{cc1}
 \eea
 Clearly, it follows from the above construction that $u^i \in {\mathbb U}_{reg}^i[0,T].$  Substituting  (\ref{cc1})  in (\ref{eq36a}) we obtain the following inequality
\begin{align}
\sum_{i=1}^N \int_{ I_{\varepsilon}} \la  {\bf \Pi}_{{\cal H}_{0,t}^{y^i}} &\Big( { H}_{u^i}(t,x^o(t),\psi^o(t),u_t^{o})\Big),v_t^i-u_t^{i,o})\ra dt
\nonumber  \\ &\geq  0, \ \forall v_t^i \in {\mathbb A}^i,  \hso i=1,2,\ldots,N.\label{eq37}
   \end{align}
   Letting $|I_{\varepsilon}|$ denote the Lebesgue  measure of the set $I_{\varepsilon}$ and dividing the above expression  by $|I_{\varepsilon}|$ and letting $\varepsilon \rightarrow 0$ we arrive at the following inequality.
\begin{align}  
 \sum_{i=1}^N  \la& {\bf \Pi}_{{\cal H}_{0,t}^{y^i}} \Big( {\mathbb H}_{u^i}(t,x^o(t),\psi^o(t),u_t^{o}\Big), v_t^i-u_t^{i,o}\ra
 \geq 0, \nonumber \\
& \hst \forall v_t^i \in {\mathbb A}^i,  \hso t \in [0,T], \hso , i=1,2,\ldots,N. \label{eq37ccc}
\end{align}
To complete the proof  of {\bf 3)}  for a given $v^i \in {\mathbb A}^i$ (deterministic) define
\begin{align}
g^i(t) \tri   \la {\bf \Pi}_{{\cal H}_{0,t}^{y^i}} &\Big( {\mathbb H}_{u^i}(t,x^o(t),\psi^o(t),u_t^{o}\Big), v^i-u_t^{i,o}\ra, \nonumber \\
& \hst   t \in [0,T], \hso , i=1,2,\ldots,N. \label{eq37c}
   \end{align}
   Then $g^i(t) \in {\cal H}_{0,t}^{y^i}$. 
We shall show that
\bea
g^i(t) \geq 0,  \hso \forall v^i \in {\mathbb A}^i, \:  \: t \in [0,T],   \: \forall i \in {\mathbb Z}_N. \label{eq37ab}
\eea
Suppose for some $i \in {\mathbb Z}_N$, (\ref{eq37ab}) does  not hold, and let $A^i \tri \{t: g^i(t)<0 \}$. Since $g^i(t) \in  {\cal H}_{0, t}^{y^i}$, $\forall t \in [0,T]$   we can choose $v_t^i$  in  (\ref{eq37ccc}) as   
$$
v_t^i \tri \left\{ \begin{array}{l} v ~\mbox{on}~ A^i \\ u_t^{i,o} \: \mbox{outside} \: A^i \end{array} \right.
$$
 together with $v_t^j =u_t^{j,o}, j \neq i, j \in {\mathbb Z}_N$. Substituting this in  (\ref{eq36a}) (with $u_t^i = v_t^i$) we arrive at   $\int_{A^i} g^i(t) dt  \geq 0,$ which contradicts the definition of $A^i$, unless $A^i$ has Lebesgue measure zero. Hence, (\ref{eq37ab}) holds which is precisely (\ref{eqh35}). This completes the derivation.}
\end{proof}

Next, we show  that the necessary conditions of optimality (\ref{eqh35}) are  also sufficient under certain  convexity conditions. 

\begin{theorem}[\textbf{Sufficient conditions for team optimality}]\label{theorem5.1s}

Consider Problem~\ref{problem1} under the conditions of Theorem~\ref{theorem5.1}, and  let $( u^o(\cdot), x^o(\cdot))$ denote any control-state pair  (decision-state)  and let $\psi^o(\cdot)$ the corresponding adjoint processes. Suppose the following conditions hold: \\
\noi \textbf{\emph{C1}} ${ H} (t, \cdot,x,u),   t \in  [0, T]$ is convex in $(x, u) \in {\mathbb R}^n \times {\mathbb A}^{(N)}$; \\
\noi \textbf{\emph{C2}} $\varphi(\cdot)$ is convex in $x \in {\mathbb R}^n$. \\
\noi Then $(u^o(\cdot),x^o(\cdot))$ is  team optimal  if it satisfies \eqref{eqh35}. In other words, necessary conditions are also sufficient.
\end{theorem}

\begin{proof}
 Let $u^o \in {\mathbb U}_{reg}^{(N)}[0,T]$ denote a candidate for the optimal team decision and $u \in {\mathbb U}_{reg}^{(N)}[0,T]$
any other decision. Then,
 \begin{align}
 J(u^o) -J(u)= & \int_{0}^{T}  \Big\{\ell(t,x^o(t),u_t^{o})  -\ell(t,x(t),u_t) \Big\} dt \nonumber \\
     &+ \Big(\varphi(x^o(T)) - \varphi(x(T))\Big)     . \label{s1}
  \end{align}
By the convexity of $\varphi(\cdot)$, we have
\bea
\varphi(x(T))-\varphi(x^o(T)) \geq \la \varphi_x(x^o(T)), x(T)-x^o(T)\ra . \label{s2}
\eea
Substituting (\ref{s2}) into (\ref{s1}) yields
 \begin{align}
 J(u^o) -J(u)\leq   & \la \varphi_x(x^o(T)),   x^o(T) - x(T)\ra \nonumber \\
 + &        \int_{0}^{T}  \Big(\ell(t,x^o(t),u_t^{o})  -\ell(t,x(t),u_t) \Big)  dt     . \label{s3}
  \end{align}
Applying the  differential rule to $\la\psi^o,x-x^o\ra$ on the interval $[0,T]$ we obtain the following equation.
\begin{align}
 \la \psi^o(T)&,  x(T) - x^o(T)\ra  \nonumber \\
 = &   \la \psi^o(0),   x(0) - x^o(0)\ra  \nonumber \\
 &+\int_{0}^{T}  \la -f_x^{*}(t,x^o(t),u_t^{o})\psi^o(t)dt  \nonumber \\
&-\ell_x(t,x^o(t),u_t^{o}), x(t)-x^o(t)\ra dt  \nonumber \\
&+   \int_{0}^{T}  \la \psi^o(t), f(t,x(t),u_t)- f(t,x^o(t),u_t^{o})\ra dt \nonumber \\
= & -   \int_{0}^{T} \la { H}_x(t,x^o(t),\psi^o(t),u_t^{o}), x(t)-x^o(t)\ra dt \nonumber \\
& +   \int_{0}^{T} \la \psi^o(t), f(t,x(t),u_t)-f(t,x^o(t),u_t^{o})\ra dt . \label{s4}
\end{align}
Note that $\psi^o(T)=\varphi_x(x^o(T))$. Substituting (\ref{s4}) into (\ref{s3}) we obtain
 \begin{align}
 J(u^o) &-J(u)  \leq    \int_{0}^{T}  \Big( {H}(t,x^o(t),\psi^o(t),u_t^{o})  \nonumber \\
 &-   {H}(t,x(t),\psi^o(t), u_t^{})\Big)dt  \nonumber \\
 &-    \int_{0}^{T} \la { H}_x(t,x^o(t),\psi^o(t),u_t^{o}), x^o(t)-x(t)\ra dt     . \label{s5}
  \end{align}
By hypothesis of convexity of $ { H}$ in $(x,u) \in {\mathbb R}^n \times {\mathbb A}^{(N)}$,  then  (\ref{s5}) reduces to
\begin{align}
  &J(u^o) -J(u)   \nonumber \\
  &\leq  \sum_{i=1}^N
       \int_{0}^{T}  <{H}_{u^i}(t,x^o(t),\psi^o(t),u_t^o), u_t^{i,o}-u_t^i> dt  \nonumber \\
       &= \sum_{i=1}^N \int_{0}^{T}  < {\bf \Pi}_{{\cal H}_{0,t}^{y^i}} \Big( {H}_{u^i}(t,x^o(t),\psi^o(t),u_t^o)\Big), u_t^{i,o}-u_t^i> dt \nonumber \\
       & \leq 0, \hst \forall u \in {\mathbb U}_{reg}^{(N)}[0,T],
       \label{s5a}
  \end{align}
where the last inequality follows from (\ref{eqh35}).    This proves that  $u^o$ optimal and hence  the necessary conditions are also sufficient.
\end{proof}

 Under the conditions  of  Theorem~\ref{theorem5.1}, it can be shown that the  necessary conditions for team optimality and PbP optimality are equivalent. Moreover, under the conditions of Theorem~\ref{theorem5.1s} it can be shown that PbP optimality implies team optimality. We state the results as a corollary.

\begin{corollary}\textbf{(Necessary and sufficient conditions for PbP optimality).} \label{corollarypbpd}
Consider the PbP optimality of  Problem~\ref{problem1} under the conditions of Theorem~\ref{theorem5.1}, \ref{theorem5.1s}.\\ The necessary and sufficient conditions for PbP optimality of $(u^o(\cdot), x^o(\cdot))$ are those of  team optimality given in Theorems~\ref{theorem5.1}, \ref{theorem5.1s}  with the variational inequality (\ref{eq16}) replaced by
\begin{align}    
  \int_0^T \la &{\bf \Pi}_{{\cal H}_{0,t}^{y^i}}\Big( { H}_{u^i}(t,x^o(t),\psi^o(t),u_t^o)\Big), u_t^i- u_t^{i,o} \ra  dt \geq 0,\nonumber \\
&  \hst \hst \hst \hst \forall u^i \in {\mathbb U}_{reg}^{i}[0,T],   \hso \forall i \in {\mathbb Z}_N. \label{eqh35i}
\end{align}
\end{corollary}

\begin{proof}
Similar to that of Theorems~\ref{theorem5.1} and \ref{theorem5.1s}.
\end{proof}

We conclude this section by stating that the team  optimality conditions, Pontryagin's maximum principle are  obtained following the  classical theory of deterministic optimal control with centralized strategies. The only variation is the characterization of the optimal strategies described by the projection of the Hamiltonian onto the Hilbert space closed subspace generated by the observables (on which the different DM actions are based on). Consequently, we state following observations.\\
\noi{\bf (O1):} 
By considering spike or needle variations, condition, the derivatives of $f$ and $\ell$ w.r.t. $u$ can be removed and replaced by $f, \ell, \varphi$ that are twice differentiable in $x \in {\mathbb R}^n$, having first partial derivatives which are measurable in $t \in [0, T]$ and continuous with respect to the rest of the arguments, and second partial derivatives which are uniformly bounded. \\
\noi{\bf (O2):} The team and PbP optimality conditions of Theorem~\ref{theorem5.1}, \ref{theorem5.1s} are based on the assumption that ${\mathbb A}^i, i=1, \ldots, N$ are convex. We can consider relaxed strategies, that is, controls which are conditional distributions, $u_t^i(d\xi | \{y^i(s): 0\leq s \leq t\}), i=1, \ldots, N$, and remove the assumptions on the differentiability of $f, \ell$ with respect to $u$, and instead assume ${\mathbb A}^i$, $i=1, \ldots, N$ are compact subsets of finite-dimensional spaces. Based on this relaxed  strategies formulation we can show existence of optimal strategies utilizing appropriate weak$^*$ topologies. Such relaxed strategies are important when the DM actions are based on a finite number of points, such as, ${\mathbb A}^{i} = \{-1, +1\}$ which is not a convex set.  \\
\noi{\bf (O3):} The team and PbP optimality conditions of Theorem~\ref{theorem5.1}, \ref{theorem5.1s} can be generalized to include pointwise and integral constraints, of $x, u$ involving inequalities and equalities. Moreover, the terminal time can be free laying on a manifold, and hence subject to optimization rather than been fixed $T$. Such problems are extensively investigated in the theory of optimal control.
Some of these problems can be transformed into  the team and PbP problems investigated earlier, by augmenting the Hamiltonian, and motifying the boundary conditions.

%
%
%
%
\section{Examples} \label{sec:examples}

In this section, we give examples for two team games with special structures, namely, Generalized Normal Form (GNF) and Linear Quadratic Form (LQF).


\subsection{Generalized Normal Form (GNF)} 
\begin{definition}[\textbf{Generalized Normal Form}]
The game is said to have {``general normal form''} if
\begin{align*}
f(t,x,u) \tri &b(t,x)+ g(t,x)u, \\
 &g(t,x)u \tri \sum_{j=1}^N g^{(j)}(t,x) u^j,  \\
 \ell(t,x,u) \tri &\frac{1}{2}\la u,R(t,x)u\ra + \frac{1}{2} \lambda(t,x)+ \la u, \eta(t,x)\ra , 
 \end{align*}
 \noi where
 \vspace{-0.4cm}
 \begin{align*}
\la u, R(t,x) u \ra \tri & \sum_{i=1}^N \sum_{j=1}^N u^{i,*} R_{ij}(t,x)u^j, \\
\la u,\eta(t,x)\ra \tri & \sum_{i=1}^N u^{i,*}\eta^i(t,x),
\end{align*}
and $R(\cdot,\cdot)$ is symmetric uniformly positive definite, and $\lambda(\cdot,\cdot)$ is uniformly positive semidefinite. 
\end{definition}
{GNF} refers to the case when the drift coefficient $f$ is linear with respect to (w.r.t.) the decision variable $u$, and the pay-off  function $\ell$ is quadratic in $u$, while $f, \ell, \varphi$ are nonlinear with respect to $x$.\\ 

\noi By the definition of Hamiltonian \eqref{h1}, its derivative is given by
\begin{align*}
{\cal H}_u(t,x,\psi,Q,u)  = &g^*(t,x) \psi  + R(t,x)u + \eta(t,x), \nonumber \\  
&\hst \hst \hst \hst (t,x)\in [0,T]\times {\mathbb R}^n.
\end{align*}
By Theorem~\ref{theorem5.1}, utilizing the fact that $u_t^{i,o} \in {\cal H}_{0,t}^{y^i}$ for each $i \in {\mathbb Z}_N$, the explicit expression for $u_t^{i,o}$ is given by
 \begin{align*}
u_t^{i,o}=&- \Big\{ \proj \Big( R_{ii}(t,x^o(t),\psi_x^o(t))  \Big)\Big\}^{-1}\\ 
&\hst \hst \hst \hst \Big\{ \proj \Big( \eta^i(t,x^o(t)) \Big) \\
& + \sum_{j=1,j \neq i}^N \proj \Big( R_{ij}(t,x^o(t))u_t^{j,o}  \Big)   \\
& +  \proj  \Big(  g^{(i),*}(t,x)\psi^{o}(t) \Big) \Big\}, \hso i=1,2, \ldots, N .
\end{align*}

\subsection{Linear Quadratic  Form (LQF)} 
\begin{definition}[\textbf{Quadratic Form}]
The game is said to have ``linear quadratic form'' if
\begin{subequations}
\begin{align}
f(t,x,u)=&A(t)x+b(t) +  B(t)u,  \label{n1} \\
 \ell(t,x)=&\frac{1}{2}\la u,R(t)u\ra  + \frac{1}{2} \la x,H(t)x\ra +\la x,F(t)\ra \nonumber \\
 &+ \la u, E(t)x\ra  +\la u,m(t)\ra ,  \label{n3}\\
  \varphi(x) =& \frac{1}{2} \la x, M(T)x \ra + \la x, N(T)\ra ,  \label{n3a}
\end{align}
\end{subequations}
and   
$R(\cdot)$ is symmetric uniformly positive definite, $H(\cdot)$ is symmetric uniformly positive semidefinite, and $M(T)$ is symmetric positive  semidefinite. 
\end{definition}

\noi From the optimal strategies under {LQF}, one obtains for $i=1,2, \ldots, N$: 
{\small 
 \begin{align*}
&u_t^{i,o}=- \Big\{  R_{ii}(t)\Big\}^{-1} \Big\{ m^i(t) +\sum_{j=1}^N E_{ij}(t) \proj \Big( x^{j,o}(t)\Big)  \\
& + \sum_{j=1,j \neq i}^N  R_{ij}  \proj \Big(u_t^{j,o}\Big) + B^{(i),*}(t)   \proj \Big(  \psi^{o}(t)\Big) \Big\}.
\end{align*}}
Note that the previous equations  can be put in the form of fixed point  matrix equation.

\subsubsection{Team games of Linear Quadratic  Form - Explicit Expressions of Adjoint Processes}

This is a necessary step before one proceeds with the computation of the explicit form of the optimal decentralized strategies, or the computation of them via fixed point methods. For  a game of LQF, let $(x^o(\cdot), \psi^o(\cdot))$ denote the solutions of the Hamiltonian system, corresponding to the optimal control $u^o$, then 
\begin{align}
\frac{d}{dt} x^o(t) &= A(t)x^o(t)+ b(t) + B(t)u_t^o,  \quad x^o(0)=x_0, \label{ex5n} \\
\frac{d}{dt}\psi^o(t) &= -A^*(t)\psi^o(t) - H(t) x^o(t) -F(t) -E^*(t) u_t^o, \nonumber  \\
 & \psi^o(T)=M(T) x^o(T)+ N(T), \label{ex6n}
\end{align}
Next, we find the form of the solution of the adjoint equation \eqref{ex6n}. Let $\{\Phi(t,s): 0\leq s \leq t \leq T\}$ denote the transition operator of $A(\cdot)$ and $\Phi^*(\cdot, \cdot)$ that of the adjoint $A^*(\cdot)$ of $A(\cdot)$. Then we have  the identity $\frac{\partial}{\partial s} \Phi^*(t,s) = -A^*(s) \Phi^*(t,s), 0 \leq s \leq t \leq T$. One can verify by differentiation that the solution $\{ \psi^o(t): t \in [0,T]\}$ of (\ref{ex6n}),  is given by  
\begin{align}
&\psi^o(t)= \Phi^*(T,t)M(T) x^o(T) + N(T) +\nonumber \\
& \int_{t}^T \Phi^*(s,t) \Big\{ H(s) x^o(s) ds +F(s)ds + E^*(s) u_s^o \Big\}  ds  \label{ex9nn}
\end{align}  
Since for any control policy, $\{x^o(s): 0\leq t \leq s \leq T\}$ is uniquely determined from \eqref{ex5n} and its current value $x^o(t)$, then (\ref{ex9nn}) can be expressed via 
\bea
\psi^o(t)=\Sigma(t) x^o(t)+ \beta^o(t) , \hst t \in [0,T], \label{ex15g}
\eea
where $\Sigma(\cdot), \beta^o(\cdot)$ determine the operators to the one expressed via (\ref{ex9nn}). \\
Next, we determine the operators $(\Sigma(\cdot), \beta^o(\cdot))$. Differentiating both sides of (\ref{ex15g}) and using (\ref{ex5n}), (\ref{ex6n}) yields
 \begin{align}
&\dot{\Sigma}(t) +A^*(t) \Sigma(t)   + \Sigma(t) A(t) + H(t)=0, \nonumber \\ &\hst  \hst \hst \hst \Sigma(T)=M(T),  \label{exg17} \\
&\dot{ \beta}^o(t) +A^*(t) \beta^o(t)+ \Sigma(t) b(t) +F(t) +\Sigma(t)  B(t)u_t^o \nonumber \\
& \hst \hst \hst + E^*(t) u_t^o=0, \hst \beta^o(T)=N(T). \label{ex18g}
\end{align}

\noi \textbf{Decentralized Information Structures} \\
Here, we invoke the minimum principle to compute the optimal strategies for team games of LQF.  Without loss of generality we assume  the  distributed  dynamical decision systems consists of  an interconnection of two subsystems,  each governed by a linear  differential equation with coupling. This can be generalized to an arbitrary number of interconnected subsystems. 

\noi Consider the distributed dynamics described below. \\
\noi {\it Subsystem Dynamics 1:} 
\begin{align}
\frac{d}{dt}x^1(t)&= A_{11}(t)x^1(t) + B_{11}(t)u_t^1+ A_{12}(t)x^2(t) \nonumber \\
& + B_{12}(t)u_t^2 , \   x^1(0)=x^1_0, \   t \in (0,T],   \label{ex30} 
\end{align}
\noi {\it Subsystem Dynamics 2:}  
\begin{align}
\frac{d}{dt}x^2(t) &=  A_{22}(t)x^2(t)  + B_{22}(t)u_t^2 + A_{21}(t)x^1(t) \nonumber \\
&+ B_{21}u_t^1 , \ x^2(0)=x^2_0, \  t \in (0,T]. \label{ex31} 
\end{align}

\noi {\bf Pay-off  Functional: }
\begin{align}
J(u^1,u^2) =& \frac{1}{2}  \Big\{ \int_{0}^T \Big[ \la \left( \begin{array}{c} x^{1}(t) \\ x^{2}(t) \end{array} \right), H(t)   \left(\begin{array}{c} x^1(t) \\ x^2(t) \end{array} \right)      \ra\nonumber \\
&+ \la \left( \begin{array}{c} u_t^{1} \\ u_t^{2} \end{array} \right), R(t)   \left(\begin{array}{c} u_t^1 \\ u_t^2 \end{array} \right) \ra         \Big]dt \nonumber \\
& + \la   \left( \begin{array}{c} x^{1}(T) \\ x^{2}(T) \end{array} \right), M(T)   \left(\begin{array}{c} x^1(T) \\ x^2(T) \end{array} \right) \ra            \Big\}.  \label{ex34}
\end{align}

\noi Define the augmented variables by 
\begin{align}
x \tri \left(\begin{array}{c} x^1 \\ x^2 \end{array} \right), \hso u \tri \left(\begin{array}{c} u^1 \\ u^2 \end{array} \right),  \hso \psi \tri \left(\begin{array}{c} \psi^1 \\ \psi^2 \end{array} \right), 
\end{align}
and matrices by
\begin{align}
 A \tri & \left[\begin{array}{cc} A_{11}  & A_{12} \\ A_{21} & A_{22} \end{array} \right], \: &B \tri \left[\begin{array}{cc} B_{11}  & B_{12} \\ B_{21} & B_{22} \end{array} \right], \: \nonumber \\[0.2cm]
  B^{(1)} \tri& \left[ \begin{array}{c} B_{11}  \\ B_{21} \end{array} \right], \: &B^{(2)} \tri \left[ \begin{array}{c} B_{12} \\ B_{22} \end{array} \right]. \nonumber \\ \nonumber
\end{align}

Let $\big(x^o(\cdot), \psi^o(\cdot)\big)$ denote the solutions of the Hamiltonian system, corresponding to the optimal control $u^o$, then 
\begin{subequations}
\begin{align}
\frac{d}{dt}x^o(t) =& A(t)x^o(t) + B(t)u_t^o , \hst x^o(0)=x_0, \label{nex5} \\
\frac{d}{dt}\psi^o(t)=& -A^*(t)\psi^o(t)   - H(t) x^o(t), \nonumber \\ 
&\psi^o(T)=M(T) x^o(T), \label{nex6} \\[0.1cm]
\psi^o(t)=& \Sigma(t) x^o(t) + \beta^o(t), \label{mv2}
\end{align}
\end{subequations}
where $\Sigma(\cdot), \beta^o(\cdot)$ are given by (\ref{exg17}), (\ref{ex18g}) with $b, F, E=0$.  
The optimal decisions $\{(u_t^{1,o}, u_t^{2,o}): 0 \leq t\leq T\}$ are given by 
\begin{subequations}
\begin{align}
\proj \Big( {\cal H}_{u^1} (t,x^{1,o}(t), & x^{2,o}(t), \psi^{1,o}(t), \psi^{2,o}(t), u_t^{1,o}, \nonumber \\ 
& u_t^{2,0})\Big) = 0,  \ t \in [0,T]. \label{ex35}
\end{align}
\begin{align}
\proj  \Big(  {\cal H}_{u^2} (t,x^{1,o}(t),& x^{2,o}(t), \psi^{1,o}(t),\psi^{2,o}(t), u_t^{1,o}, \nonumber \\ 
&u_t^{2,0}) \Big) = 0, \ t \in [0,T]. \label{ex36}
\end{align}
\end{subequations}
From (\ref{ex35}), (\ref{ex36}) the optimal decisions for $\ t \in [0,T]$ are given by
\begin{subequations}
\begin{align}
u_t^{1,o} =& -R_{11}^{-1}(t)   B^{(1),*}(t)  {\bf \Pi}_{{\cal H}_{0,t}^{y^1}} \Big(\psi^{o}(t)\Big) \nonumber \\ 
&\hst -R_{11}^{-1}(t)R_{12}(t) {\bf \Pi}_{{\cal H}_{0,t}^{y^1}} \Big(u_t^{2,o} \Big) ,  \label{ex41}
\end{align}
\begin{align}
u_t^{2,o} = &-R_{22}^{-1}(t) B^{(2),*}(t) {\bf \Pi}_{{\cal H}_{0,t}^{y^2}} \Big(\psi^{o}(t)\Big) \nonumber \\ 
&\hst -R_{22}^{-1}(t)R_{21}(t) {\bf \Pi}_{{\cal H}_{0,t}^{y^2}}\Big(u_t^{1,o} \Big). \label{ex42}
\end{align}
\end{subequations}
One can proceed further to utilize the solution for $\psi^o(\cdot)$ to express the projections in (\ref{ex41}), (\ref{ex42}) into projections of the state $x^o(\cdot)$ onto the subspaces ${\cal H}_{0,t}^{y^i}, i=1, 2$, and then find the equations governing these projections.  This procedure is lenghty and hence it is omitted. 

%
%
%
%
\section{Discrete-time dynamical systems}\label{sec:discrete}

By either discretizing the continuous-time system (or considering the discrete-time analog), we have the Hamiltonian at each time step $k$ given by
\bes
{\cal H} (k,x,\psi,u) \tri    \langle f(k,x,u),\TC{\psi(k+1)} \rangle  + \ell(k,x,u), 
\ees
where $k\in \mathbb{Z}_{T-1}^0$ and $T$ is a positive integer. Note that the adjoint is one step ahead of the other terms. In terms of the Hamiltonian, the state process satisfies the differential equation
\begin{subequations}
\begin{align}
{x}(&k+1) = f(k,x(k),u_k) \nonumber \\ 
&= {\cal H}_\psi (k,x(k),\TC{\psi(k+1)},u_k),   \ k\in \mathbb{Z}_{T-1}^0 \label{st1id}\\
        x(0) &=  x_0. \label{st1jad}
\end{align}
\end{subequations}
For any $u \in {\mathbb U}_{reg}^{(N)}[0,T]$, the adjoint process is $\psi \in  \ell^2([0,T], {\mathbb R}^n)$ satisfies the following backward differential equation
\begin{subequations}
\bea
{\psi} (k) = - { H}_x (k,x(k),\TC{\psi(k+1)},u_k),  \ k\in \mathbb{Z}_{T-1}^0. \label{adj1ad} \\
\hst \psi(T) =   \varphi_x(x(T)). \label{eq18d}
 \eea
 \end{subequations}
The process ${\psi}^o$  is the  unique  solution of the backward difference equation (\ref{adj1ad}), (\ref{eq18d}) and   $u^o \in {\mathbb U}_{reg}^{(N)}[0,T]$ satisfies  the  inequalities:
\begin{align}
   \la {\bf \Pi}_{{\cal H}_{0,k}^{y^i}}& \Big( H_{u^i}(k,x^o(k),\psi^o(k+1),u_k^{o})\Big) , v^i-u_k^{i,o}\ra   \geq 0, \nonumber \\
   &\hst \forall v^i \in {\mathbb A}^i,  \ k \in [0,T], i=1,2,\ldots, N.   \label{eqh35ii}
\end{align}

\section{Conclusions and Future Work} \label{sec:conclusions}

In this paper we have considered team games for distributed decision systems, with decentralized information patterns for each DM. Necessary and sufficient optimality conditions with respect to  team optimality and PbP optimality criteria are derived, based on Pontryagin's maximum principle. 
The methodology is very general, and applicable to many areas. However, several additional issues remain to be investigated. Below, we provide a short list. \\
\noi \textbf{(F1)} The derivation of optimality conditions can be used in other type of games such as Nash-equilibrium games with decentralized information structures for each DM, and minimax games of robust control. \\
\noi \textbf{(F2)} The methodology can be extended to deal with exogenous inputs in the state dynamics and the  measurements, by assuming these belong to $L^2$ or $\ell^2$ spaces. For distributed systems of control of linear quadratic form, with decentralized information structures, one may also invoke the  minimax formulation found  in \cite{hassibi} which invokes Krein spaces, instead of Hilbert spaces. 

\bibliographystyle{IEEEtran}
\bibliography{bibdata}

\begin{thebibliography}{10}
\providecommand{\url}[1]{#1}
\csname url@samestyle\endcsname
\providecommand{\newblock}{\relax}
\providecommand{\bibinfo}[2]{#2}
\providecommand{\BIBentrySTDinterwordspacing}{\spaceskip=0pt\relax}
\providecommand{\BIBentryALTinterwordstretchfactor}{4}
\providecommand{\BIBentryALTinterwordspacing}{\spaceskip=\fontdimen2\font plus
\BIBentryALTinterwordstretchfactor\fontdimen3\font minus
  \fontdimen4\font\relax}
\providecommand{\BIBforeignlanguage}[2]{{%
\expandafter\ifx\csname l@#1\endcsname\relax
\typeout{** WARNING: IEEEtran.bst: No hyphenation pattern has been}%
\typeout{** loaded for the language `#1'. Using the pattern for}%
\typeout{** the default language instead.}%
\else
\language=\csname l@#1\endcsname
\fi
#2}}
\providecommand{\BIBdecl}{\relax}
\BIBdecl

\bibitem{bamieh-voulgaris2005}
B.~Bamieh and P.~Voulgaris, ``A convex characterization of distributed control
  problems in spatially invariant systems with communication constraints,''
  \emph{Systems and Control Letters}, vol.~54, no.~6, pp. 575--583, 2005.

\bibitem{nayyar-mahajan-teneketzis2011}
A.~Nayyar, A.~Mahajan, and D.~Teneketzis, ``Optimal control strategies in
  delayed sharing information structures,'' \emph{IEEE Transactions on
  Automatic Control}, vol.~56, no.~7, pp. 1606--1620, 2011.

\bibitem{vanschuppen2011}
J.~H. van Schuppen, ``Control of distributed stochastic systems-introduction,
  problems, and approaches,'' in \emph{International Proceedings of the IFAC
  World Congress}, 2011.

\bibitem{lessard-lall2011}
L.~Lessard and S.~Lall, ``A state-space solution to the two-player optimal
  control problems,'' in \emph{Proceedings of 49th Annual Allerton Conference
  on Communication, Control and Computing}, 2011.

\bibitem{mahajan-martins-rotkowitz-yuksel2012}
A.~Mahajan, N.~Martins, M.~Rotkowitz, and S.~Yuksel, ``Information structures
  in optimal decentralized control,'' in \emph{In Proceedings of the 51st
  Conference on Decision and Control (CDC)}, 2012.

\bibitem{charalambous-ahmedFIS_Parti2012}
\BIBentryALTinterwordspacing
C.~D. Charalambous and N.~U. Ahmed, ``Centralized versus decentralized team
  games of distributed stochastic differential decision systems with noiseless
  information structures-{P}art {I}: Applications,'' \emph{Submitted to IEEE
  Transactions on Automatic Control}, pp. 1--39, February 2013. [Online].
  Available: \url{http://arxiv.org/abs/1302.3452}
\BIBentrySTDinterwordspacing

\bibitem{charalambous-ahmedFIS_Partii2012}
\BIBentryALTinterwordspacing
------, ``Centralized versus decentralized team games of distributed stochastic
  differential decision systems with noiseless information structures-{P}art
  {II}: Applications,'' \emph{Submitted to IEEE Transactions on Automatic
  Control}, pp. 1--39, February 2013. [Online]. Available:
  \url{http://arxiv.org/abs/1302.3416}
\BIBentrySTDinterwordspacing

\bibitem{marschak1955}
J.~Marschak, ``Elements for a theory of teams,'' \emph{Management Science},
  vol.~1, no.~2, 1955.

\bibitem{radner1962}
R.~Radner, ``Team decision problems,'' \emph{The Annals of Mathematical
  Statistics}, vol.~33, no.~3, pp. 857--881, 1962.

\bibitem{waal-vanschuppen2000}
P.~R. Wall and J.~H. van Schuppen, ``A class of team problems with discrete
  action spaces: Optimality conditions based on multimodularity,'' \emph{SIAM
  Journal on Control and Optimization}, vol.~38, no.~3, pp. 875--892, 2000.

\bibitem{hassibi}
H.~Babak, S.~A. H., and K.~Thomas, \emph{Indefinite-Quadratic Estimation and
  Control: A Unified Approach to $\mathcal{H}_2$ and $\mathcal{H}_\infty$
  Theories}.\hskip 1em plus 0.5em minus 0.4em\relax Society for Industrial and
  Applied Mathematics, 1999.

\end{thebibliography}
\end{document}